\documentclass{article}%
\usepackage{amsfonts}
\usepackage{amsmath}
\usepackage{amssymb}
\usepackage{graphicx}%
\providecommand{\U}[1]{\protect\rule{.1in}{.1in}}
\numberwithin{equation}{section}
\newtheorem{theorem}{Theorem}[section]

\newtheorem{lemma}[theorem]{Lemma}

\newenvironment{proof}[1][Proof]{\noindent\textbf{#1.} }{\ \rule{0.5em}{0.5em}}
\begin{document}

\title{Sch\"{u}tt's theorem for vector-valued sequence spaces}
\author{D. E. Edmunds
\and Yu. Netrusov}
\date{ Accepted for Journal of Approximation Theory}
\maketitle

\begin{abstract}
The entropy numbers of certain finite-dimensional operators acting between
vector-valued sequence spaces are estimated, thus providing a generalisation
of the famous result of Sch\"{u}tt. In addition, two-sided estimates of the
entropy numbers of some diagonal operators are obtained.

\end{abstract}

\section{\bigskip Introduction}

The notion of the entropy of a set and the companion idea of the entropy
numbers of a bounded linear map between (quasi-) Banach spaces are now of
proven importance in analysis, especially in spectral theory and approximation
theory: one has only to think of the ground-breaking work of Kolmogorov and
Tikhomirov \cite{KT}, the subsequent study related to Hilbert's thirteenth
problem by Vitushkin and Henkin \cite{VH}, and Birman and Solomyak's
celebrated paper \cite{BiS} on embeddings of Sobolev spaces to have an idea of
the possibilities. The theorem of Sch\"{u}tt mentioned in the title relates to
the natural embedding $id$ of $l_{p}^{m}$ in $l_{q}^{m},$ where $n\in
\mathbb{N}$ and $1\leq p<q\leq\infty:$ it asserts that given any
$k\in\mathbb{N}$, there are positive constants $c_{1},c_{2},$ independent of
$m$ and $n,$ such that the $n^{th}$ entropy number $e_{n}(T)$ of $T$ satisfies
$c_{1}A(m,n)\leq e_{n}(id)\leq c_{2}A(m,n),$ where $A(m,n)$ is an explicit
function of $m$ and $n$ (see Theorem 2.1 below). \ This was proved in
\cite{Sch} by means of volume arguments. Now it is known that the result holds
whenever $0<p<q\leq\infty:$ the upper estimate was obtained in \cite{ET},
Proposition 3.2.2, again by volume arguments, while for the lower estimate we
refer to \cite{EN}, Theorem 2 and \cite{K}. Apart from its intrinsic interest,
a good deal of the importance of Sch\"{u}tt's theorem stems from its
connection with embeddings of function spaces. In the work of Birman and
Solomyak alluded to above, estimates of the entropy numbers of embeddings
between Sobolev spaces were obtained by means of piecewise polynomial
approximations. To deal with more general spaces, such as those of Besov
(perhaps with generalised smoothness) or Lizorkin-Triebel type, it is more
convenient to use decompositions of wavelet (see, for example, \cite{Dau},
\cite{Mey}, \cite{Tri2}) or atomic (see \cite{HN}) form to reduce questions of
embeddings of function spaces to considerations of mappings between sequence
spaces. It is in connection with these mappings that the Sch\"{u}tt result
plays a part.

In this paper we obtain two-sided estimates for the entropy numbers of certain
mappings between vector-valued sequence spaces. More precisely, we consider a
mapping
\[
T:l_{p}^{m}\left(  \left\{  X_{i}\right\}  _{i=1}^{m}\right)  \rightarrow
l_{q}^{m}\left(  \left\{  Y_{i}\right\}  _{i=1}^{m}\right)  ,
\]
where $0<p<q\leq\infty,$ the $X_{k}$ and $Y_{k}$ are quasi-Banach spaces and
$T$ is defined by $Tx=\left(  T_{1}x_{1},...,T_{m}x_{m}\right)  ,$ where
$x=\left(  x_{1},...,x_{m}\right)  ,$ each $T_{i}$ being a bounded linear map
from $X_{i}$ to $Y_{i}.$ Our main focus is on the case when $X=X_{1}%
=...=X_{m},$ $Y=Y_{1}=...=Y_{m},$ $T_{i}=\lambda_{i}T_{0}$ $(i=1,...,m),$
where $T_{0}:X\rightarrow Y$ is a bounded linear operator and the $\lambda
_{i}$ are real numbers. In particular, when $\lambda_{i}=1$ for all
$i\in\{1,...,m\}$ it is shown that knowledge of the entropy numbers
$e_{1}(T_{0}),...,e_{n}(T_{0})$ of the operator $T_{0}$ leads to two-sided
estimates of the entropy numbers $e_{n}(T)$ $(n\in\mathbb{N})$ of $T.$ In
\cite{EN} we gave a generalisation of Sch\"{u}tt's theorem to the case of
finite-dimensional spaces with symmetric bases: in the present paper we use
some ideas from \cite{EN} but in a very simple form. Unlike the volume
arguments mentioned above, and the interpolation techniques appearing in
\cite{GL} and \cite{K} (in \cite{K} the same ideas as in \cite{EN} were
used-see Lemma \ref{Lemma 2.6} below-but with functions with values in the set
$\left\{  -1,0,1\right\}  $ instead of the characteristic functions of
\cite{EN}), our proofs are essentially combinatoric in nature: by
specialisation they give an independent proof of Sch\"{u}tt's theorem.

For previous work on mappings between vector-valued sequence spaces we refer
to \cite{KS}, \cite{CK} and the references contained in these papers. Interest
in the entropy numbers of embeddings of function spaces owes much to
\cite{BiS}, in which Sobolev spaces were considered; since the appearance of
\cite{BiS} the literature on the subject has grown enormously. Many papers
deal with estimates of the entropy numbers of embeddings of Besov spaces with
generalised smoothness; we refer again to \cite{CK}, \cite{EN2}, \cite{Tri2}
and the references given in those works.\

\section{Preliminaries}

\subsection{Background}

Throughout the paper $\log$ is to be understood as $\log_{2},$ $\left[
x\right]  $ will denote the integer part of the real number $x,$
$\mathbb{N}_{0}=\mathbb{N}\cup\{0\}$ and $a\asymp b$ means that $c_{1}a\leq
b\leq c_{2}a$ for some positive constants independent of variables occurring
in $a$ and $b.$ Given quasi-Banach spaces $X$ and $Y,$ we shall write $B(X,Y)$
for the space of all bounded linear maps from $X$ to $Y$, abbreviating this to
$B(X)$ when $X=Y;$ the closed unit ball in $X$ will be denoted by $B_{X}$ and
the quasinorm on $X$ by $\left\Vert \cdot\right\Vert _{X}.$ We recall that a
quasi-Banach space $Z$ is said to be an $r-$Banach space if the quasi-norm
$\left\Vert \cdot\right\Vert _{Z}$ has the property that for all $z_{1}%
,z_{2}\in Z,$%
\[
\left\Vert z_{1}+z_{2}\right\Vert _{Z}^{r}\leq\left\Vert z_{1}\right\Vert
_{Z}^{r}+\left\Vert z_{2}\right\Vert _{Z}^{r};
\]
the quasi-norm is then said to be an $r-$norm. It is well known (see, for
example, \cite{Aok} and \cite{Rol}) that if $Z$ is any quasi-Banach space then
there exist $r\in(0,]$ and an $r-$norm on $Z$ equivalent to the original quasi-norm.
Given a finite set $A$ we shall write $\sharp A$ for the cardinality of the set $A.$

Let $n\in\mathbb{N}$ and suppose that $M$ is a bounded subset of an $r-$normed
quasi-Banach space $Y.$ The $n^{th}$ (dyadic) outer entropy number $e_{n}(M)$
of $M$ is defined to be the infimum of those $\varepsilon>0$ such that $M$ can
be covered by $2^{n-1}$ balls in $Y$ of radius $\varepsilon.$ The $n^{th}$
outer entropy number of a map $T\in B(X,Y)$ (where $X$ and $Y$ are
quasi-Banach spaces) is
\[
e_{n}(T):=e_{n}\left(  T\left(  B_{X}\right)  \right)  ;
\]
the numbers $e_{n}(T)$ are monotonic decreasing as $n$ increases, with
$e_{1}(T)=\left\Vert T\right\Vert ;$ and $T$ is compact if and only if
$\lim_{n\rightarrow\infty}e_{n}(T)=0.$ Moreover, for all $s,n\in\mathbb{N},$
and whenever $T_{1}+T_{2}$ and $R\circ S$ are properly defined bounded linear
operators acting between quasi-Banach spaces,%
\[
e_{s+n-1}\left(  R\circ S\right)  \leq e_{s}\left(  R\right)  e_{n}\left(
S\right)
\]
and, if the target space of $T_{1}$ and $T_{2}$ is an $r-$Banach space,
\[
e_{s+n-1}^{r}\left(  T_{1}+T_{2}\right)  \leq e_{s}^{r}\left(  T_{1}\right)
+e_{n}^{r}\left(  T_{2}\right)  .
\]
Following Pietsch (\cite{Pie}, 12.1.6), for each $n\in\mathbb{N}$ we denote by
$f_{n}(T)$ the (dyadic) inner entropy number of $T\in B(X,Y),$ defined to be
the supremum of all those $\varepsilon>0$ such that there are $x_{1}%
,...,x_{2^{n-1}+1}\in B_{X}$ with $\left\Vert Tx_{i}-Tx_{j}\right\Vert
_{Y}\geq2\varepsilon$ whenever $i,j$ are distinct points of $\left\{
1,2,...,2^{n-1}+1\right\}  .$ If $Y$ is an $r-$Banach space, then the outer
and inner entropy numbers are related by%
\[
f_{n}(T)\leq2^{1/r-1}e_{n}(T)\leq2^{1/r}f_{n}(T)\text{ \ }(n\in\mathbb{N}).
\]
These estimates were proved by Pietsch in the Banach space case ($r=1$); the
proof in the general case merely involves a simple modification of his arguments.

We shall need vector-valued versions of the familiar sequence space $l_{p}$
and its $m-$dimensional subspace $l_{p}^{m}.$ Let $p\in(0,\infty
],m\in\mathbb{N}$ and suppose that $X_{1},...,X_{m}$ are quasi-Banach spaces.
Then
\[
l_{p}^{m}\left(  \left\{  X_{i}\right\}  _{i=1}^{m}\right)  :=\left\{
x=(x_{1},...,x_{m}):x_{i}\in X_{i}\text{ for each }i\right\}  ;
\]
(for simplicity we shall denote this space by $l_{p}^{m}\left(  X_i\right)  $)  endowed with the quasi-norm%
\begin{align*}
\left\Vert x\mid l_{p}^{m}\left(  \left\{  X_{i}\right\}  _{i=1}^{m}\right)
\right\Vert  &  :=\left(  \sum\limits_{i=1}^{m}\left\Vert x_{i}\right\Vert
_{X_{i}}^{p}\right)  ^{1/p}\text{ if }0<p<\infty,\\
\left\Vert x\mid l_{\infty}^{n}\left(  \left\{  X_{i}\right\}  _{i=1}%
^{m}\right)  \right\Vert  &  :=\sup_{1\leq i\leq m}\left\Vert x_{i}\right\Vert
_{X_{i}},
\end{align*}
it is a quasi-Banach space. When $X_{1}=...=X_{m}=X,$ we shall simply denote
this space by $l_{p}^{m}\left(  X\right)  .$

The theorem of Sch\"{u}tt in which we are interested appears in \cite{Sch} and
asserts the following:

\begin{theorem}
\label{Theorem 2.1} Let $m,n\in\mathbb{N}$ and $1\leq p<q\leq\infty;$ denote
by $id$ the natural embedding of $l_{p}^{m}$ in $l_{q}^{m}.$ Then there are
positive constants $c_{1},c_{2},$ independent of $m$ and $n$, such that
\[
c_{1}A(m,n)\leq e_{n}(id)\leq c_{2}A(m,n),
\]
where
\[
A(m,n)=\left\{
\begin{array}
[c]{ccc}%
1, & \text{if} & n\leq\log m,\\
\left(  \frac{\log(m/n+1)}{n}\right)  ^{1/p-1/q}, & \text{if} & \log m\leq
n\leq m,\\
2^{-n/m}m^{1/q-1/p}, & \text{if} & n\geq m.
\end{array}
\right.
\]

\end{theorem}

Various authors have contributed to the generalisation of this result to the
case $0<p<q\leq\infty$. For the estimate from above, we refer to
\cite{ET}, Proposition 3.2.2; an elementary proof of the lower estimate in the
case $\log m\leq n\leq m$ is given in Theorem 2 of \cite{EN}, where a
generalisation of Sch\"{u}tt's result for the case of quasinormed spaces with a
symmetric basis was\ presented (such a generalisation is still unknown for the
case $n\geq m);$ a proof of the lower estimate contained in \cite{K}. More
detailed estimates of the constants (upper and lower) and a new proof of the
whole result for $0<p<q\leq\infty$ are given in \cite{GL}.

\subsection{Preparatory results}

Here we present the main ingredients to be used in the proof of the main result.

\begin{lemma}
\label{Lemma 2.2} Let $m\in\mathbb{N}.$ Then there is a set $\Gamma
(m)\subset(0,1]^{n}$ with the following properties:

\noindent(i) $\sharp\Gamma(m)\leq2^{5m/2}.$

\noindent(ii) For any sequence $\left\{  \varepsilon_{i}\right\}  _{i=1}^{m}$
in $\Gamma(m),$ the numbers $n\varepsilon_{i}$ are positive integers for all
$i\in\{1,2,...,n\},$ $\sum\limits_{i=1}^{m}\varepsilon_{i}\leq3$ and for all
$t>0,$%
\[
\sharp\left\{  i\in\{1,2,...,m\}:\varepsilon_{i}\geq t\right\}  \leq2/t.
\]

\noindent(iii) For any sequence $\left\{  \alpha_{i}\right\}  _{i=1}^{m}$ with
each $\alpha_{i}\in\lbrack0,1]$ and $\sum\limits_{i=1}^{m}\alpha_{i}=1$ there
is a sequence $\left\{  \varepsilon_{i}\right\}  _{i=1}^{m}\in\Gamma(m)$ such
that $\alpha_{i}\leq\varepsilon_{i}$ for all $i\in\{1,2,...,m\}.$
\end{lemma}

\begin{proof}
Put%
\[
E=\left\{  2^{k}/m:k\in\mathbb{N}_{0}, 2^{k}<m \right\}  \cup\left\{1\right\}
\]
and define $\Gamma(m)$ to be the set of all sequences $\left\{  \varepsilon
_{i}\right\}  _{i=1}^{m}\in E^{m}$ such that $\sum\limits_{i=1}^{m}%
\varepsilon_{i}\leq3$ and $\sharp\left\{  i\in\{1,2,...,m\}:\varepsilon
_{i}\geq t\right\}  \leq2/t$ for all $t>0.$ This ensures that (ii) holds. To
estimate the number of elements in $\Gamma(m)$ we observe that if $\left\{
\varepsilon_{i}\right\}  _{i=1}^{m}\in\Gamma(m),$ $k\in\mathbb{N}_{0},$
$A(k):=\left\{  i\in\{1,2,...,m\}:\varepsilon_{i}=2^{k}/m\right\}  $ and
$B(k):=\left\{  i\in\{1,2,...,m\}:\varepsilon_{i}\geq2^{k}/m\right\}  ,$
then%
\[
\sharp A(k)\leq\sharp B(k)\leq\min\left(  m,m/2^{k-1}\right)  .
\]
Fix the sets $A(0),A(1),A(2)$ and $B(3);$ note that $\sharp B(3)\leq m/4.$
Then for the choice of the sets $A(k)$ $(k\geq3)$ there are at most
\[
\prod\limits_{k=3}^{\infty}2^{m/2^{k-1}}=2^{m/2}%
\]
possibilities. Since  $A(k)\subset B(k),$ $\sharp B(k)\le m/2^{k-1},$ and
it follows that once the sets $A(0), \dots A(k-1)$ have been chosen,  there are at most   $2^{m/2^{k}}$ possibilities for the choice of
$A(k,$  $k\ge3.$

We now claim that given any non-empty finite set $S$ with $m$ elements, there
are $2^{2m}$ distinct representations of $S$ as the union of $4$ disjoint
subsets. For if $S=\{s_{1},...,s_{m}\}$ and $S=\cup_{j=1}^{4}S_{j},$ where the
$S_{j}$ are disjoint, then each $s_{i}$ has to belong to some $S_{j},$ and as
there are $4$ choices for each $s_{i},$ the total number of choices is
$4^{m}=2^{2m}.$ Thus $\sharp\Gamma(n)\leq2^{2m}\cdot2^{m/2}=2^{5m/2},$ so that
(i) holds. The final property (iii) is established in a routine fashion and is
left to the reader.
\end{proof}

\begin{lemma}
\label{Lemma 2.3}Let $m\in\mathbb{N}\backslash\{1\},$ suppose that
$0<p<q\leq\infty$ and put $\alpha=1/p-1/q.$ For each $i\in\{1,2,...,m\}$ let
$X_{i},Y_{i}$ be quasi-Banach spaces and $T_{i}\in B\left(  X_{i}%
,Y_{i}\right)  .$ Suppose that for every $i,s$ $\in\{1,2,...,m\},$%
\[
e_{s}\left(  T_{i}\right)  \leq(m/s)^{\alpha}.
\]
Let $T:l_{p}^{m}\left(  X_{i}\right)  \rightarrow
l_{q}^{m}\left( Y_{i}\right)  $ be the linear
operator defined by
\[
T(x)=\left(  T_{1}(x_{1}\right)  ,...,T_{m}(X_{m})),\text{ \ }x=(x_{1}%
,...,x_{m})\in X_{1}\times X_{2}\times...X_{m}.
\]
Then%
\[
e_{5m}(T)\leq3^{1/q}.
\]

\end{lemma}

\begin{proof}
Let $W=$ $l_{p}^{m}\left(  X_{i}\right)  .$ Given
any point $x\in B_{W},$ there is a sequence $\left\{  \alpha_{i}\right\}
_{i=1}^{m}$ with each $\alpha_{i}\in\lbrack0,1]$ and $\sum\limits_{i=1}%
^{m}\alpha_{i}=1$ such that $x\in\prod\limits_{i=1}^{m}\alpha_{i}^{1/p}B_{X_{i}%
}.$ By Lemma \ref{Lemma 2.2}, it follows that
\[
B_{W}\subset\bigcup\left(  \prod\limits_{i=1}^{m}\varepsilon_{i}^{1/p}%
B_{X_{i}}\right)  ,
\]
where the union is taken over all sequences $\left\{  \varepsilon_{i}\right\}
_{i=1}^{m}\in\Gamma(m),$ where $\Gamma(m)$ is the set defined in Lemma
\ref{Lemma 2.2}. Viewing
\[
K:=\prod\limits_{i=1}^{m}\varepsilon_{k}^{1/p}T_{i}\left(  B_{X_{i}}\right)
.
\]
as a subset of $l_{q}^{m}\left(  \left\{  Y_{i}\right\}  _{i=1}^{m}\right)  ,$
we estimate $e_{2n+1}(K).$ Put $m_{i}=m\varepsilon_{i}$ $(i=1,2,...,m).$ Then
\[
e_{m_{i}}\left(  \varepsilon_{i}^{1/p}T_{i}\left(  B_{X_{i}}\right)  \right)
\leq\varepsilon_{i}^{1/p}\left(  \frac{n}{n\varepsilon_{i}}\right)  ^{\alpha
}=\varepsilon_{i}^{1/q}.
\]
Since $\sum\limits_{i=1}^{m}(m_{i}-1)\leq3m-m=2m,$ application of the
following simple lemma, the proof of which is omitted, shows that
\[
e_{2m+1}\left(  K\right)  \leq\left(  \sum\limits_{i=1}^{m}\varepsilon
_{i}\right)  ^{1/q}\leq3^{1/q}.
\]
As $\sharp\Gamma(m)\leq2^{5m/2},$ we see that $e_{N}(T)\leq3^{1/q},$ where
$N=\left[  \frac{5m}{2}\right]  +2n+1:$ the result follows.
\end{proof}

\begin{lemma}
\label{Lemma 2.4}Let $m,n\in\mathbb{N}$ and let $n_{1},...,n_{m}$ be
non-negative integers such that $n-1=\sum\limits_{i=1}^{m}(n_{i}-1);$ let
$q\in(0,\infty].$ For each $i\in\left\{  1,2,...,m\right\}  $ suppose that
$Z_{i},Y_{i}$ are quasi-Banach spaces and $U_{i}\in B(Z_{i},Y_{i}).$ Let
$U:l_{\infty}^{m}\left( Z_{i}\right)
\rightarrow l_{q}^{m}\left(   Y_{i}\right)  $ be
the linear operator defined by%
\[
U(z)=\left(  U_{1}(z_{1}),...,U_{m}(z_{m})\right)  ,\text{ }z=(z_{1}%
,...,z_{m})\in Z_{1}\times...\times Z_{m}.
\]
Then
\[
e_{n}(U)\leq\left(  \sum\limits_{i=1}^{m}e_{n_{i}}^{q}(U_{i})\right)  ^{1/q}.
\]

\end{lemma}

In the next section we shall need the following estimates, proved in
\cite{EN1} (or \cite{EN}) and \cite{EN2}.

\begin{lemma}
\label{Lemma 3.1}

\noindent(i) If $k,m\in\mathbb{N},$ $k\leq m,$ then
\[
\left(  \frac{m}{k}\right)  ^{k}\leq\binom{m}{k}\leq\left(  \frac{em}%
{k}\right)  ^{k}.
\]

\noindent(ii) There are positive constants $c_{1},c_{2}$ such that for any
$m,n\in\mathbb{N}$ with $2\leq n\leq m\leq2^{n}$ the following estimates hold:%
\[
2^{c_{1}n}\leq\binom{m}{k}\leq2^{c_{2}n},
\]
where $k$ is the smallest positive integer such that
\[
k\geq A:=\frac{n}{2\log\left(  2m/n\right)  }.
\]

\end{lemma}

\begin{proof}
As (i) is well known we simply deal with (ii) and suppose that $m\geq2n.$ By
(i) we have%
\begin{align*}
\frac{\log\binom{m}{k}}{n}  &  \asymp\frac{k\log(m/k)}{n}\asymp\frac
{A\log\left(  m/A\right)  }{n}\\
&  =\frac{\log(2m/n)-\log\log(2m/n)}{2\log(2m/n)}.
\end{align*}
The rest follows easily.
\end{proof}

\begin{lemma}
\label{Lemma 2.5}Let $m,n\in\mathbb{N},$ $b\in(0,\infty)$ and $0<p<q\leq
\infty;$ put $\alpha=1/p-1/q.$ For each $i\in\{1,2,...,m\}$ let $X_{i},Y_{i}$
be quasi-Banach spaces and $T_{i}\in B(X_{i},Y_{i}).$ Let $T:l_{p}\left(
\left\{  X_{i}\right\}  _{i=1}^{m}\right)  \rightarrow$\ $l_{q}\left(
\left\{  Y_{i}\right\}  _{i=1}^{m}\right)  $ \ be the linear operator defined
by
\[
T(x)=\left(  T_{1}(x_{1}),T_{2}(x_{2}),...,T_{m}(x_{m})\right)  ,\text{
}x=(x_{1},x_{2},...,x_{m})\in X_{1}\times X_{2}\times...\times X_{m}.
\]
Suppose that $f_{n}(T_{i})\geq b$ for each $i\in\{1,2,...,m\}.$ Then%
\[
f_{k}(T)\geq2^{-1/q}bm^{-\alpha},\text{ where }k=[((n-1)m/6].
\]

\end{lemma}

\begin{lemma}
\label{Lemma 2.6} Let $E$ be a set, let $v\in\mathbb{N}$ be such that
$64e^{3}v\leq\sharp E$ and put $\mathcal{L}(E,v)=\left\{  E_{1}\subset
E:\sharp E_{1}=v\right\}  .$ Then there is a set $\mathcal{L}(E,v,1/2)\subset
\mathcal{L}(E,v)$ with the following properties:

\noindent(i) for any distinct $E_{1},E_{2}\in\mathcal{L}(E,v,1/2),$%
\[
\sharp\left(  E_{1}\cap E_{2}\right)  \leq v/2;
\]

\noindent(ii)%
\[
\left(  \sharp\mathcal{L}(E,v,1/2)\right)  ^{4}\geq\sharp\mathcal{L}%
(E,v)=\binom{\sharp E}{v}.
\]

\end{lemma}

\section{ The main results}

\begin{theorem}
\label{Theorem 3.2} Let $0<p<q\leq\infty,$ set $\alpha=1/p-1/q$ and let
$m,n\in\mathbb{N},$ $2\leq n\leq m\leq2^{n}.$ Let $X,Y$ be $r-$normed
quasi-Banach spaces and suppose that $T_{0}\in B(X,Y).$ Let $T(m):l_{p}%
^{m}(X)\rightarrow l_{q}^{m}(Y)$ be the linear operator defined by
$T(m)(x)=\left(  T_{0}(x_{1}),...,T_{0}(x_{m})\right)  ,x=\left(
x_{1},...,x_{m}\right)  \in l_{p}^{m}(X).$ Then
\begin{equation}
c_{1}A(n,m,T_{0})\leq e_{n}(T(m))\leq c_{2}A(n,m,T_{0}).\label{Eq 3.1}%
\end{equation}
Here $c_{1},c_{2}$ are positive constants which depend on the parameters $p,q$
and $r$ only, and
\[
A(n,m,T_{0})=\max\left(  \left\Vert T_{0}\right\Vert \left(  \frac
{\log(m/n)+1}{n}\right)  ^{\alpha},\max_{k\in\{1,2,...,n\}}\left(
(k/n)^{\alpha}e_{k}(T_{0})\right)  \right)  .
\]

\end{theorem}

\begin{proof}
First note that given any $a>1,$ there are positive constants $C_{1}%
(a),C_{2}(a)$ such that, for any $m,n,\widetilde{m},\widetilde{n}\in
\mathbb{N}$ with $m\geq n,\widetilde{m}\geq\widetilde{n},ma\geq\widetilde
{m}\geq m/a,na\geq\widetilde{n}\geq n/a,$
\begin{equation}
C_{1}(a)A(\widetilde{n},\widetilde{m},T_{0})\geq A(n,m,T_{0})\geq
C_{2}A(\widetilde{n},\widetilde{m},T_{0}).\label{Eq 3.2}%
\end{equation}
Now we show that the required statement is a consequence of the following two assertions.

\noindent1. There are positive constants $C_{3},C_{4}$ and an integer $a>1$
such that, for any $m,n\in\mathbb{N}$ with $m\geq n,$ the following estimates
hold:%
\[
e_{na}\left(  T(m)\right)  \leq C_{3}A(n,m,T_{0}),\text{ }f_{n(a)}(T(m))\geq
C_{4}A(n,m,T_{0}).
\]
Here $n(a)$ is the smallest positive integer greater than or equal to $n/a.$

\noindent2. Given any integer $b\geq1$ (we need this assertion for $b=1$
only)$,$ there is a constant $C_{5}(b)=C_{5}$ such that for every
$n\in\mathbb{N},$%
\[
f_{nb}(T(n))\geq C_{5}A(n,n,T_{0}).
\]

Indeed let us prove that the first estimate in (\ref{Eq 3.1}) is a consequence
of the first estimate in assertion 1 and the estimates in (\ref{Eq 3.2}). Let
$\widetilde{m},\widetilde{n}\in\mathbb{N}$ with $\widetilde{m}\geq
\widetilde{n};$ without loss of generality we can suppose that $\widetilde
{n}\geq2a,a\geq2.$ Choose $n\in\mathbb{N}$ in such a way that $na\leq
\widetilde{n}\leq(n+1)a.$ Then%
\begin{align*}
e_{\widetilde{n}}(T(\widetilde{m})  & \leq e_{na}(T(\widetilde{m})\leq
C_{3}A(n,\widetilde{m},T_{0})\leq C_{1}C_{3}A(na,\widetilde{m},T_{0})\\
& \leq C_{1}^{2}C_{3}A(\widetilde{n},\widetilde{m},T_{0}).
\end{align*}
\ To prove the second estimate in (\ref{Eq 3.1}), once more let $\widetilde
{m},\widetilde{n}\in\mathbb{N}$ with $\widetilde{m}\geq\widetilde{n}.$ Choose
$n\in\mathbb{N}$ in such a way that $n(a)\geq\widetilde{n}\geq n(a)-1;$
without loss of generality we can suppose that $\widetilde{n}\geq2a,a\geq2.$
There are two possibilities: $n\leq\widetilde{m}$ or  $n\geq\widetilde{m}.$ In
the first case we use the estimates%
\[
f_{\widetilde{n}}(T(\widetilde{m}))\geq f_{n(a)}(T(\widetilde{m}))\geq
C_{4}A(n,\widetilde{m},T_{0})\geq C_{2}C_{4}A(\widetilde{n},\widetilde
{m},T_{0}).
\]
In the second case we use the estimate $f_{n}(T(\widetilde{m}))\geq
f_{\widetilde{m}}(T(\widetilde{m})),$ and then assertion 2 with $b=1$ and the
estimate (\ref{Eq 3.1}).

Now we prove assertions 1 and 2. We begin with the proof of the upper estimate
in statement 1 and let $k$ be the positive integer defined in Lemma
\ref{Lemma 3.1} (ii). For any set $F\subset\{1,2,...,m\}$ such that $\sharp
F=k$ let $T(m)_{F}:$ $l_{p}^{m}(X)\rightarrow l_{q}^{m}(Y)$ be the linear
operator defined by%
\[
T(m)_{F}(x)=\left(  \chi_{F}(1)T_{0}(x_{1}),...,\chi_{F}(m)T_{0}%
(x_{m})\right)  ,\text{ }x=(x_{1},...,x_{m})\in l_{p}^{m}(X).
\]
Here $\chi_{F}$ is the characteristic function of $F.$ Let $s\in
\mathbb{N},\varepsilon>0$ and $\eta=e_{s}(T(m));$ let $B_{p}$ be the unit ball
in $l_{p}^{m}(X)$ and denote by $\Gamma(F)$ an $(\eta+\varepsilon)-$net (of
cardinality $2^{s-1}$) of $T(m)_{F}(B_{p})\subset$ $l_{q}^{m}(Y)$ such that
$y_{i}=0$ for any $i\in\{1,2,...,m\}\backslash F$ and $y=(y_{1},...,y_{m}%
)\in\Gamma(F).$ Let $\Gamma=\cup\Gamma(F),$ where the union is taken over all
sets $F\subset\{1,2,...,m\}$ with $\sharp F=k.$ Then
\[
\sharp \Gamma(F)\leq2^{s-1}\binom{m}{k}.
\]
Much as in \cite{EN1} (see the proof of Lemma 11) it can be seen that $\Gamma$
is an $\varepsilon_{0}-$net of $T(m)(B_{p})$ in $l_{q}^{m}(Y),$ where
\[
\varepsilon_{0}^{r}=(\eta+\varepsilon)^{r}+\left(  \left\Vert T_{0}\right\Vert
/(k+1)^{\alpha}\right)  ^{r}.
\]

Now let $x=(x_{1},...,x_{m})\in B_{p}$ and let $F$ be any subset of
$\{1,2,...,m\}$ such that $\sharp F=k$ and $\left\Vert x_{i}\right\Vert
_{X}\geq\left\Vert x_{j}\right\Vert _{X}$ whenever $i\in F$ and $j\in$
$\{1,2,...,m\}\backslash F.$ Then $\left\Vert T_{0}x_{j}\right\Vert _{Y}%
\leq\left\Vert T_{0}\right\Vert /(k+1)^{1/p}$ if $j\in$
$\{1,2,...,m\}\backslash F.$ By H\"{o}lder's inequality,%
\[
\left\Vert \left(  T(m)-T(m)_{F}\right)  (x)\right\Vert _{l_{q}^{m}(Y)}%
\leq\left\Vert T_{0}\right\Vert /(k+1)^{\alpha}.
\]
In view of Lemma \ref{Lemma 3.1} these arguments imply that there is a
positive integer $C_{6}$ such that, for any $s\in\mathbb{N}$,%
\[
e_{C_{6}n+s}(T(m))\leq2^{1/q}\max\left(  e_{s}(T(k),\left\Vert T_{0}%
\right\Vert /(k+1)^{\alpha}\right)  .
\]
Together with Lemma \ref{Lemma 2.3} this gives the required upper estimate in
statement 1. The lower estimate is a consequence of Lemma \ref{Lemma 2.5}. \

To prove statement 2, note that because of Lemma \ref{Lemma 2.5} and the
estimates of $A(n,m,T_{0})$ given in (\ref{Eq 3.1}), it is enough to show that
given any $b\in\mathbb{N},$ there is a positive constant $C_{7}(b)=C_{7}$ such
that for every $n\in\mathbb{N}$,%
\[
f_{nb}(T(n))\geq C_{7}\left\Vert T_{0}\right\Vert /n^{\alpha}.
\]
Let $n,u\in\mathbb{N}$ with $n>64e^{3},$ put $E=\{1,2,...,n\},$ suppose that
$v$ is the largest integer such that $64e^{3}v\leq n,$ and let $x\in X$
satisfy $\left\Vert x\right\Vert _{X}\leq1$ and $\left\Vert T_{0}x\right\Vert
_{Y}\geq\left\Vert T_{0}\right\Vert /2.$ Define $I(u)$ to be the subset of the
unit ball of $l_{p}^{m}(X)$ consisting of all points with $i^{th}$ coordinate
$(1\leq i\leq m)$
\[
\sum\limits_{j=1}^{u}2^{-kr}v^{-1/p}\chi_{E(j)}(i)x\text{ for some }%
E(j)\in\mathcal{L}(E,v,1/2).
\]
Then $\sharp I(u)\geq2^{C_{8}n}$ and $\left\Vert T(m)x-T(m)y\right\Vert
_{l_{q}^{m}(Y)}\geq2^{-ru}v^{-\alpha}C_{9}$ for all distinct $x,y\in I(u).$
The result follows.
\end{proof}

To conclude we formulate one more result, the proof\ of which is similar to
that of the last theorem.

\begin{theorem}
\label{Theorem 3.3} Let $0<p<q\leq\infty,$ set $\alpha=1/p-1/q$ and let
$m,n\in\mathbb{N},$ $m\leq2^{n}.$ For each $i\in\{1,2,...,m\}$ let
$X_{i},Y_{i}$ be $r-$normed quasi-Banach spaces and suppose that $T_{i}\in
B(X_{i},Y_{i}).$ Let $T:l_{p}^{m}(X_{i})\rightarrow l_{q}^{m}(Y_{s})$
be the linear operator defined by
\[
T(x)=\left(  T_{1}(x_{1}),...,T_{m}(x_{m})\right)  x=(x_{1}%
,...,x_{m})\in l_{p}^{m}(X_{i}).
\]
(i) Let $m\geq2n$ and suppose that $\left\Vert T_{1}\right\Vert \geq\left\Vert
T_{2}\right\Vert \geq...\geq\left\Vert T_{m}\right\Vert ,$ $\left\Vert
T_{1}\right\Vert \leq2\left\Vert T_{n}\right\Vert ;$ put%
\begin{align*}
A(n,m)  &  =\max_{s\in\{n,n+1,...,m\},s\leq 2^n}\left\Vert T_{s}\right\Vert \left(
\frac{\log(2s/n)}{n}\right)  ^{\alpha},\\
B(n,m)  &  =\max_{k\in\{1,2,...,n\},\text{ }i\in\{1,2,...,m\}}\left(
(k/n)^{\alpha}e_{k}\left(  T_{i}\right)  \right)  .
\end{align*}
Then
\[
c_{1}A(n,m)\leq e_{n}(T)\leq c_{2}\max(A(n,m),B(n,m)),
\]
where $c_{1},c_{2}$ are positive constants which depend on the parameters $p$
and $q$ only.

\noindent(ii) Suppose that $m\leq n$ and $T_{1}=T_{2}=...=T_{m}=T_{0}.$ For
any $a>0$ let
\[
D(a,n,m)=\max_{k\in\{1,2,...,n\},\text{ }k\geq a}\left(  (k/n)^{\alpha}%
e_{k}\left(  T_{0}\right)  \right)  .
\]
Then%
\[
c_{5}D(c_{3}n/m,n,m)\leq e_{n}(T)\leq c_{6}D(c_{4}n/m,n,m),
\]
where $c_{3},c_{4}$ are absolute constants and the constants $c_{5},c_{6}$
depend on the parameters $p$ and $q$ only.
\end{theorem}

\

\bigskip

\bigskip D. E. Edmunds, Department of Mathematics, University of Sussex,
Brighton BN1 9QH, U.K.

Yu. Netrusov, Department of Mathematics, University of Bristol, Bristol BS8
1TW, U.K.

\

\begin{thebibliography}{99}                                                                                               %


\bibitem {Aok}Aoki, T., Locally bounded linear topological spaces, Proc. Imp.
Acad. Tokyo \textbf{18 (}1942), 588-594.

\bibitem {BiS}Birman, M. S. and Solomyak, M. Z., Piecewise polynomial
approximation of functions of the class $W_{p}^{\alpha}$, Mat. Sb. \textbf{73
}(115) (1967), 331-355; English translation in Math. USSR Sb. (1967), 295-317.

\bibitem {CK}Cobos, F. and K\"{u}hn, T., Approximation and entropy numbers in
Besov spaces of generalised smoothness, J. Approx. Theory \textbf{160 }(2009), 56-70.

\bibitem {Dau}Daubechies, I., \textit{Ten lectures on wavelets}, CBMS-NSF
Regional Conf. Series Appl. Math., SIAM. Philadelphia, 1992.

\bibitem {EN}Edmunds, D. E. and Netrusov, Yu., Entropy numbers of embeddings
of Sobolev spaces in Zygmund spaces, Studia Math. \textbf{128} (1998), 71-102.

\bibitem {EN1}Edmunds, D. E. and Netrusov, Yu., Entropy numbers and
interpolation, Math. Ann. \textbf{351} (2011), 963-977.

\bibitem {EN2}Edmunds, D. E. and Netrusov, Yu., Entropy numbers of operators
acting between vector-valued sequence spaces, Math. Nachr., to appear.

\bibitem {ET}Edmunds, D. E. and Triebel, H., \textit{Function spaces, entropy
numbers, differential operators, }Cambridge Univ. Press, Cambridge, 1996.

\bibitem {GL}Guedon, O. and Litvak, A. E., \textit{Euclidean projections of a
}$p-$\textit{convex body}, Vol. \textbf{1745,} Lecture Notes in Mathematics,
Springer, 2000, pp.95-108.

\bibitem {HN}Hedberg, L. I. and Netrusov, Yu., An axiomatic approach to
function spaces, spectral synthesis, and Luzin approximation, Amer. Math. Soc
Memoirs \textbf{188 }(207), no. 882.

\bibitem {KT}Kolmogorov, A. N. and Tikhomirov, V. M., $\varepsilon-$entropy
and $\varepsilon-$capacity of sets in functional spaces, Uspekhi Mat. Nauk
\textbf{14} (2) (1959), 3-86 (Russian); English transl. Amer. Math. Soc.
Transl. Ser. 2, \textbf{17} (1961), 277-364. \

\bibitem {K}K\"{u}hn, T., A lower estimate for entropy numbers, J. Approx.
Theory \textbf{110} (2001), 120-124.

\bibitem {KS}K\"{u}hn, T. and Schonbek, T. P., Entropy numbers of diagonal
operators between vector-valued sequence spaces, J. London Math. Soc. (2)
\textbf{64} (2001), 739-754.

\bibitem {LT}Lindenstrauss, J. and Tzafriri, L., \textit{Classical Banach
spaces I and II,} Springer-Verlag, Berlin-Heidelberg-New York, 1977 and 1979.

\bibitem {Mey}Meyer, Y., \textit{Wavelets and operators, }Cambridge Univ.
Press, Cambridge, 1992.

\bibitem {Pie}Pietsch, A., \textit{Operator ideals, }North-Holland, Amsterdam, 1980.

\bibitem {Rol}Rolewicz, S., On a certain class of linear metric spaces, Bull.
Acad. Polon. Sci Cl. III 5 (1957), 471-473.

\bibitem {Sch}Sch\"{u}tt, C., Entropy numbers of diagonal operators between
symmetric Banach spaces, J. Approx. Theory \textbf{40} (1984), 121-128.\

\bibitem {Tri2}Triebel, H., Function spaces and wavelets on domains, European
Math. Soc., Z\"{u}rich, 2008.

\bibitem {VH}Vitushkin, A. G. and Henkin, G. M., Linear superposition of
functions, Uspekhi Mat. Nauk \textbf{22 }(1967), 77-124 (Russian); English
transl. Russian Math. Surveys \textbf{22 }(1967), 77-125.
\end{thebibliography}
\end{document}